\documentclass[11pt]{article}

\usepackage{graphicx}
\usepackage{url}
\usepackage{hyperref}

\usepackage{amsmath}
\usepackage{amssymb}
\usepackage{amsfonts}
\usepackage{amsthm}
\usepackage{latexsym}

\newtheorem{theorem}{Theorem}

\newtheorem{corollary}{Corollary}
\newtheorem{remark}{Remark}
\newtheorem{proposition}{Proposition}

\title{SHAPE AND TOPOLOGY OPTIMIZATION FOR VARIATIONAL INEQUALITIES
    WITH POINTWISE BOUNDARY OBSERVATION}

\author{Cornel Marius MUREA\footnote{{\tt cornel.murea@uha.fr}
University of Haute Alsace, Dept. of Mathematics, Mulhouse, France.} 
\and Dan TIBA\footnote{{\tt dan.tiba@imar.ro}
Inst. of Mathematics, Romanian Academy, Bucharest, Romania; Paper written with
financial support of ECOMATH France-Roumanie.}
}
\date{~}

\begin{document}

\maketitle

\begin{abstract}
  We study  optimal design problems involving variational inequalities with
  unilateral conditions in the domain and pointwise boundary observation.
  We use regularizing and penalization tehniques in the setting of the
  Hamiltonian approach to shape/topology optimization problems.
  Numerical examples are also included.
\end{abstract}



\section{Introduction}

We consider a fixed domain approach, based on level functions for the (unknown) geometry and on
the penalization/regularization of the state system. The shape and topology optimization problem
is associated to nonlinear second order elliptic equations with Dirichlet boundary conditions.
The optimal design problem can be approximated, via penalization, by an elliptic optimal control
problem (see \cite{c8}).

Our methodology  is underpinned by  functional variations \cite{npt}, \cite{MT2022}.
For linear and nonlinear state systems, it is known
that both closing and
creating new holes is possible. Moreover, we prove differentiability and we apply standard gradient
algorithms in the optimization process, that are known to converge locally, \cite{ber}.
For a detailed discussion of optimal control problems governed by nonlinear and nonsmooth state systems,
see \cite{book}, \cite{tiba}.

In the setting of  functional variations, one assumes that all the admissible domains $\Omega$ are
contained in a given bounded domain $D \subset \mathbb{R}^2$. We denote by $\mathcal{O}$
their family. We assume that 
this is  generated by a family $\mathcal{F} \subset \mathcal{C}(\overline{D})$ of
admissible
level functions, as open   sublevel sets:
\begin{equation}\label{eq:1.1}
\Omega = \Omega_g = {\rm int}\{ \mathbf{x} \in D: \; g(\mathbf{x}) \leq 0 \}, \quad  g \in \mathcal{F}.
\end{equation}

We recall  that functional variations have the form
\begin{equation}\label{eq:1.3}
g + \lambda h, \; \lambda \in \mathbb{R}, \; g,h \in \mathcal{F},
\end{equation}
which enables the work in functional spaces instead of the  very  specialized geometric variations.

To generate domains, we add to (\ref{eq:1.1}) a condition that
selects just the component satisfying ($\mathbf{x}_0 \in D$ is fixed)
\begin{equation}\label{eq:1.2}
\mathbf{x}_0 \in \partial\Omega_g, \quad \forall g \in \mathcal{F}.
\end{equation}

Notice that, for any admissible $g \in \mathcal{F}$, we have $g(\mathbf{x}_0) = 0$.

A general  shape and topology optimization problem has the form
(for $\Omega_g  \in \mathcal{O}$ defined as above):
\begin{eqnarray}
&&\min_{g \in \mathcal{F}} J(g, y_g), \label{eq:1.4}
\\
&& A y_g = f \quad {\rm in} \; \Omega_g, \label{eq:1.5}
\\
&& By_g = 0 \quad {\rm on} \; \partial\Omega_g. \label{eq:1.6}
\end{eqnarray}

In (\ref{eq:1.4}) - (\ref{eq:1.6}), $A$ is an elliptic partial differential
operator, $B$ is a
boundary operator, defining the state system and $J$ is some cost functional.
The operators may be linear or nonlinear and appropriate assumptions will be imposed as
necessity appears.

Assume now that $\mathcal{F} \subset \mathcal{C}^1(\overline{D})$ and
\begin{equation}\label{eq:1.7}
  \nabla g(\mathbf{x}) \neq 0, \; \forall g \in \mathcal{F},\;
  \forall  \mathbf{x} \in D,\;with\; g(\mathbf{x}) = 0 .
\end{equation}

Assume, as well, that
\begin{equation}\label{eq:1.11}
g > 0 \quad {\rm on} \; \partial D, \quad \forall g \in \mathcal{F},
\end{equation}
which ensures that $\partial\Omega_g \cap \partial D = \emptyset$.
If, moreover, the open cone $\mathcal{F} \subset \mathcal{C}^2(D)$   is given by (\ref{eq:1.7}),
(\ref{eq:1.11}),  then the solution of the Hamiltonian system below is periodic and  gives a global
parameterization of $\partial \Omega_g$

\begin{eqnarray}
z_1^\prime(t) & = & -\frac{\partial g}{\partial x_2}\left(z_1(t),z_2(t)\right),\quad t\in I_g,
\label{1:2.9}\\
z_2^\prime(t) & = &  \frac{\partial g}{\partial x_1}\left(z_1(t),z_2(t)\right),\quad t\in I_g,
\label{1:2.10}\\
\left(z_1(0),z_2(0)\right)& = & \mathbf{x}_0
\in \partial\Omega_g,
\label{1:2.11}
\end{eqnarray}

\noindent
by \cite{MT2022}, \cite{HMT2024}. We may take $I_g = [0, T_g]$,
where $T_g > 0$ denotes
the principal period in (\ref{1:2.9})-(\ref{1:2.11}). Moreover,
the dependence $g \in \mathcal{F} \to T_g$ has differentiability properties,  due to \cite{MT2022}
if the family of shape functions satisfies $\mathcal{F} \subset \mathcal{C}^2(D)$ and \ref{eq:1.7}),
(\ref{eq:1.11}).

In every $\Omega \in \mathcal{O}$, $\Omega = \Omega_g \subset D$, $g \in \mathcal{F}$,
we consider the
variational inequality:
\begin{eqnarray}
-\Delta y_g + \beta(y_g) & \ni & f \quad {\rm in} \; \Omega_g,  \label{eq:2.1}
\\
y_g &=& 0 \quad {\rm on} \; \partial\Omega_g,\label{eq:2.2}
\end{eqnarray}
where $f \in L^2(D)$ is given and $\beta \subset \mathbb{R}\times \mathbb{R}$ is a
maximal monotone graph satisfying
\begin{equation}\label{eq:2.3}
0 \in \beta(0).
\end{equation}

The nonlinear boundary value problem (\ref{eq:2.1}), (\ref{eq:2.2}) includes as
an example the obstacle
problem with obstacle $\varphi = 0$. In case the obstacle $\varphi : D \to \mathbb{R}$ is
some other
function (negative in $D$, to ensure compatibility with (\ref{eq:2.2}) for any
$\Omega \in \mathcal{O}$), then $\beta(y_g - \varphi)$ should be written in (\ref{eq:2.1}).

It is known,  \cite{6b}, that (\ref{eq:2.1}), (\ref{eq:2.2})
has a unique solution
$y_g \in H^2(\Omega_g) \cap H_0^1(\Omega_g)$ with $\beta(y_g) \in L^2(\Omega_g)$, if $\partial\Omega_g$
is in
$\mathcal{C}^{1,1}$.
Therefore the formulation (\ref{eq:1.4}), (\ref{eq:2.1}), (\ref{eq:2.2})
makes sense.
The regularity of
$\partial\Omega_g$ may be obtained from the implicit functions theorem since 
$\mathcal{F} \subset \mathcal{C}^2(D)$
and (\ref{eq:1.7}) is satisfied (we get even $\mathcal{C}^2$ regularity for
$\partial\Omega_g$).

We penalize and regularize (\ref{eq:2.1}) as follows ($\varepsilon > 0$ is ``small''):
\begin{eqnarray}
  -\Delta y_{\varepsilon} + \beta_{\varepsilon}(y_{\varepsilon})
  + \dfrac{1}{\varepsilon} H_{\varepsilon}(g)y_{\varepsilon}
  &=& f \quad {\rm in} \; D, \label{eq:2.4}
\\
y_{\varepsilon} &=& 0 \quad {\rm on} \; \partial D. \label{eq:2.5}
\end{eqnarray}

Above, $\beta_{\varepsilon}$ combines the Yosida approximation $\beta^{\varepsilon}$ with
Friedrichs mollifiers:
\begin{equation}\label{eq:2.6}
\beta_{\varepsilon}(y) = \int_{-\infty}^{\infty}\beta^{\varepsilon}(y - \varepsilon s)\rho(s)ds, 
\end{equation}
with $\rho \in \mathcal{C}_{0}^{\infty}(R)$, $\mathrm{supp}(\rho) \subset [-1, 1]$,
$\rho(-s) = \rho(s)$, $\rho(s) \geq 0$ and
$\int_{-\infty}^{\infty}\rho(s)ds = 1$.

Concerning $H(\cdot)$, it denotes the Heaviside function on $\mathbb{R}$ and its
maximal monotone extension in
$\mathbb{R} \times \mathbb{R}$ as well (i.e., we take $H(0) = [0,1]$).
  The notation $H_{\varepsilon}$ is obtained from
$H \subset \mathbb{R} \times \mathbb{R}$ again via Yosida
approximation and Friedrichs mollifiers, as in (\ref{eq:2.6}). Clearly, (\ref{eq:2.4}),
(\ref{eq:2.5})
has a unique solution $y_{\varepsilon} \in H^2(D) \cap H_0^1(D)$ if $\partial D$ is supposed to be of class
$\mathcal{C}^{1,1}$. 

This distributed penalization technique, using $H, H_{\varepsilon}$, is due to Natori and Kawarada
\cite{Kawarada1981}, in the setting
of free boundary problems, and employed in shape optimization, for instance, in \cite{npt},
 for linear elliptic equations.

 We indicate an approximation result, following \cite{SEMA_2024}, under weak regularity
 assumptions (we do not require $g\in \mathcal{F}$, for the moment) on the domain $\Omega$
 (class $\mathcal{C}$, i.e., continuous boundary or the segment property, see \cite{Adams1975},
 Def. 4.2, p. 66, \cite{ti2013}). We recall here that $\Omega$ is of class $\mathcal{C},$ if 
there exists a locally finite open cover of $\partial \Omega$ and an associated sequence of vectors
$\{y_j\}$, not null, such that for any $x \in \partial \Omega$ there is $j$ with $x + ty_j \in \Omega$
for $0<t<1$ and this remains valid on the associated neighborhood of $x$, in the open cover of
$\partial \Omega$. 

\medskip

\begin{theorem}\label{ch3_2:theo:2.1}
If $\Omega = \Omega_g$ is of class $\mathcal{C}$, then, on a subsequence,
  $y_{\varepsilon}|_{\Omega_g} \to y_g$ weakly in $H^1(\Omega_g)$ and $\Delta y_{\varepsilon} \to \Delta y_g$,
  $\beta_{\varepsilon}(y_{\varepsilon}) \to \beta(y_g)$ weakly in $L^2(\Omega_g)$.
If $g \in \mathcal{F}$, then
$\frac{\partial y_{\varepsilon}}{\partial \mathbf{n}} \to \frac{\partial y_g}{\partial \mathbf{n}}$
in $H^{-1/2}(\partial\Omega_g)$ weakly.
\end{theorem}

\medskip

In the next section we discuss a special choice of the cost functional (\ref{eq:1.4}),
namely in the so called pointwise observation case, that describes realistic, but computationally
difficult situations.

\section{Pointwise observation}

In this section, an important role in the justification of the following computations and arguments
is played by the topological stability result from \cite{tiba2023}, for small functional
variational perturbations. Somewhat paradoxical, this stability result allows topological
changes in optimal design since the corresponding steps in the descent gradient algorithms
are not small, in general.

A natural and maybe the easiest case of pointwise boundary observation (taking into acount
(\ref{1:2.9})-(\ref{1:2.11})) is the following performance index:
\begin{equation}\label{eq:2.2.1}
  \min_{g \in \mathcal{F}} \left\{ J_1
  = \left|\frac{\partial y_{\varepsilon}}{\partial \mathbf{n}}(\mathbf{x}_0) - \alpha\right|^2\right\},
\end{equation}
where $\alpha \in \mathbb{R}$ is  given, $y_{\varepsilon}$ is the solution of (\ref{eq:2.4}),
(\ref{eq:2.5}) and $\mathbf{x}_0 \in \partial\Omega$ is the initial
condition in (\ref{1:2.9})-(\ref{1:2.11}). Other points may be taken into account as well,
see (\ref{eq:2.2.2}). The normal derivative refers to the domain $\Omega_g$ and can be
expressed as $\frac{\nabla  y_{\varepsilon} \cdot \nabla g}{| \nabla g |}$ in any point of $D$.
Notice that  (\ref{eq:2.2.1}) makes sense due to

\begin{remark}\label{r1}
  If $ \mathcal{F} \subset \mathcal{C}^2$ and $f \in H^1 (D)$, then  $y_{\varepsilon} \in H^3 (D) $ since
  $\beta_{\varepsilon}(y_{\varepsilon})$ and $H_{\varepsilon}(y_{\varepsilon})$ are in $H^2 (D)$ due to the
  regularity of $\beta_{\varepsilon}, H_{\varepsilon}$ and we
  know that $y_{\varepsilon} \in H^2 (D)$. Consequently, $\Delta y_{\varepsilon} \in H^1 (D)$ and we get
  $y_{\varepsilon} \in H^3 (D)$, by standard
  regularity results for the Laplace equation and  $y_{\varepsilon} \in \mathcal{C}^1 (\overline{D})$ by
  the Sobolev theorem. That is $\frac{\partial y_{\varepsilon}}{\partial \mathbf{n}}$ has a
  pointwise sense.
\end{remark}

We have denoted by $T_g$ the main period of the solution of (\ref{1:2.9})-(\ref{1:2.11})) corresponding
to $g \in \mathcal{F}$. Let $l$ be a given natural number and divide $[0,T_g]$ (and implicitly
$\partial \Omega_g$) in $l$  parts starting from the initial condition (\ref{1:2.11})
(that we may arbitrarily choose on $\partial \Omega_g$, in this case). The described process is
like a discretization procedure. We define in this way a preassigned number of observation points
on $\partial\Omega_g$ and another example of pointwise cost functional is
\begin{equation}\label{eq:2.2.2}
  \min_{g \in \mathcal{F}} \left\{J_2 =
  \sum_{i=0}^{l-1} \left|\frac{\partial y_{\varepsilon}}{\partial \mathbf{n}}
  \left(\mathbf{z}\left(\frac{iT_g}{l}\right)\right)
  - \alpha_i\right|^2\right\},
\end{equation}
where $\alpha_i \in \mathbb{R}$ are given. A variant of (\ref{eq:2.2.2}) is to choose the
observation points just on some part of $\partial\Omega_g$ , by fixing the indices $l_1 \leq i \leq l_2$,
with $l_1, l_2$ given natural numbers, in the corresponding sum (or by using another appropriate subset
of indices).

Of course, one can define pointwise observation functionals, by using points in $\Omega$, not
on $\partial\Omega$ (and under the additional condition that $g \in \mathcal{F}$ are negative in
such points). But this type of observation is less realistic, in general (it is difficult to place sensors inside a body) and we shall consider just the minimization of
tracking type functionals as in (\ref{eq:2.2.1}), (\ref{eq:2.2.2}).

In the sequel, we study the differentiability properties of (\ref{eq:2.2.1}), (for (\ref{eq:2.2.2})
the arguments are quite similar). We recall first the following two results from \cite{MT2019},
\cite{Tiba2013}, \cite{MT2022}, \cite{SEMA_2024}. We consider the perturbations of the parameters
(controls) $g+\lambda h$ and the system below  (\ref{3.7})-(\ref{3.10}) characterizes the
variations of the corresponding solutions of
the state system (\ref{eq:2.4}), (\ref{eq:2.5}), (\ref{1:2.9})-(\ref{1:2.11}). Furthermore,
$\mathbf{z}_{\lambda}$ denotes the solution of (\ref{1:2.9})-(\ref{1:2.11}) associated to
$g+\lambda h$, $T_\lambda$ denotes the main period of $\mathbf{z}_{\lambda}$ and we shall also
write $y_{\varepsilon}(g), y_{\varepsilon}(g+\lambda h)$ when we emphasize the dependence on the
level functions describing the geometry.

\begin{proposition}\label{prop:3.2}
  The system in variations corresponding to (\ref{eq:2.4}),
  (\ref{eq:2.5}),\\
(\ref{1:2.9})-(\ref{1:2.11}) is:
\begin{eqnarray}
  -\Delta u + \beta^\prime_{\varepsilon}(y_{\varepsilon})u + \dfrac{1}{\varepsilon}H_{\varepsilon}(g)u
  + \dfrac{1}{\varepsilon}H^\prime_{\varepsilon}(g)y_{\varepsilon}h &=& 0 \quad in \; D,
  \label{3.7}\\
u &=& 0 \quad on \; \partial D. \nonumber 
\end{eqnarray}
\begin{eqnarray}
w_1^\prime& = & -\nabla\partial_2 g(\mathbf{z})\cdot \mathbf{w}
-\partial_2 h(\mathbf{z}),\quad\hbox{in } [0,T_g],
\label{3.8}\\
w_2^\prime& = & \nabla\partial_1 g(\mathbf{z})\cdot \mathbf{w}
+\partial_1 h(\mathbf{z}),\quad\hbox{in } [0,T_g],
\label{3.9}\\
w_1(0)& = &0,\ w_2(0) = 0,
\label{3.10}
\end{eqnarray}
where $u=\lim_{\lambda\rightarrow 0}\frac{y_\varepsilon (g + \lambda h)-y_\varepsilon (g)}{\lambda}$, 
$\mathbf{w}=[w_1,w_2]=\lim_{\lambda\rightarrow 0}
\frac{\mathbf{z}_{\lambda} - \mathbf{z}}{\lambda}$
and the limits exist in $H^2 (D)$, 
respectively $\mathcal{C}^1([0,T_g])^2$. 
\end{proposition}

\begin{proposition}\label{prop:3.3}
  Under the assumptions (\ref{eq:1.7}), (\ref{eq:1.11}), if $T_\lambda$ denotes the main period of
  the solution of (\ref{1:2.9})-(\ref{1:2.11}), we have:
$$
\lim_{\lambda\rightarrow 0}\frac{T_\lambda-T_g}{\lambda}
=-\frac{w_2(T_g)}{z_2^\prime (T_g)}
$$
if $z_2^\prime (T_g) \neq 0$.
\end{proposition}

\begin{remark}\label{rem:3.2}
If $z_1^\prime (T_g)\neq 0$, the limit is
$-\frac{w_1(T_g)}{z_1^\prime (T_g)}$. In general, we denote by
$\theta(g,r)$ this limit. Either the last condition in Proposition \ref{prop:3.3} or this condition here
follow by (\ref{eq:1.7}). In case both conditions are valid, it is known that the two mentioned
results coincide (this is a consequence of Prop.1 in \cite{ti2018}).
\end{remark}  

The  directional derivatives of cost functionals can be computed as follows.

We consider the cost (\ref{eq:2.2.1}) and we notice (under the notations of Prop.\;\ref{prop:3.2} and
by Rem.\ref{r1}) that the computations to obtain (\ref{3.7}) can be slightly refined as follows
($\varepsilon$ is fixed here):
\begin{eqnarray}
&& -\Delta(y_\varepsilon (g + \lambda h) - y_{\varepsilon}(g))
  + \beta_{\varepsilon}(y_\varepsilon (g + \lambda h))
  - \beta_{\varepsilon}(y_{\varepsilon}(g))
  \nonumber \\
&& + \dfrac{1}{\varepsilon}H_{\varepsilon}(g + \lambda h)(y_\varepsilon (g + \lambda h) - y_{\varepsilon}(g))
\nonumber \\
&& - \dfrac{1}{\varepsilon}\left[ H_{\varepsilon}(g + \lambda h) - H_{\varepsilon}(g) \right]y_{\varepsilon}(g) = 0.
\label{eq:2.88}
\end{eqnarray}
Dividing (\ref{eq:2.88}) by $\lambda \neq 0$ and using Prop.\ref{prop:3.2}, we get
$\Delta \frac{y_\varepsilon (g + \lambda h)-y_\varepsilon (g)}{\lambda}$ bounded in $H^1 (D)$ due to
the regularity properties of $\beta_{\varepsilon}, H_{\varepsilon}$, that is
$\frac{y_\varepsilon (g + \lambda h)-y_\varepsilon (g)}{\lambda} \rightarrow u$ strongly in 
$\mathcal{C}^1 (D)$, due to the Sobolev theorem.  Then, we can compute the derivative of the
pointwise cost (\ref{eq:2.2.1}). First, we notice that:
\begin{eqnarray}
  L &=& \lim_{\lambda \rightarrow 0}\frac{1}{\lambda}
  \left[\left|\frac{\partial y_{\varepsilon}(g + \lambda h)}{\partial \mathbf{n}} ( \mathbf{x}_0) - \alpha\right|^2
    -  \left|\frac{\partial y_{\varepsilon}(g)}{\partial \mathbf{n}} ( \mathbf{x}_0) - \alpha\right|^2\right]
  \nonumber \\
  &=&  2\left(\frac{\partial y_{\varepsilon}(g)}{\partial \mathbf{n}}( \mathbf{x}_0)
  - \alpha\right)\frac{\partial  u}{\partial \mathbf{n}}(\mathbf{x}_0).
\label{eq:2.99}
\end{eqnarray}

In (\ref{eq:2.99}), we have decomposed the first paranthesis in the product of the sum by the
difference of the squared terms and we  apply Prop.\ref{prop:3.2} and the above argument,
to pass to the limit. We continue with the other term:
\begin{equation}
\lim_{\lambda \rightarrow 0}\frac{1}{\lambda}
\left[\left|\frac{\partial y_{\varepsilon}(g + \lambda h)}{\partial \mathbf{n}_{\lambda}} ( \mathbf{x}_0) - \alpha\right|^2
-\left|\frac{\partial y_{\varepsilon}(g + \lambda h)}{\partial \mathbf{n}} ( \mathbf{x}_0) - \alpha\right|^2\right].
  \label{eq:2.77}
\end{equation}
where $n_{\lambda}$ is the normal in $\mathbf{x}_0$ to $\partial \Omega_{\lambda}$, $\Omega_{\lambda}$ being
the domain associated to $g+\lambda h$ via (\ref{eq:1.1}), (\ref{eq:1.2}). After the decomposition of
the difference of the squares, the important part to be investigated is:
\begin{eqnarray}
&& \lim_{\lambda \rightarrow 0}\frac{1}{\lambda}
  \left[\frac{\partial y_{\varepsilon}(g + \lambda h)}{\partial \mathbf{n}_{\lambda}}( \mathbf{x}_0)
    - \frac{\partial y_{\varepsilon}(g + \lambda h)}{\partial \mathbf{n}} ( \mathbf{x}_0)\right]
  \nonumber \\
&=&  \lim_{\lambda \rightarrow 0}\frac{1}{\lambda}
\left[\frac{\nabla y_{\varepsilon}(g + \lambda h) \cdot \nabla(g+\lambda h)}{|\nabla(g+\lambda h)|}
(\mathbf{x}_0) - \frac{\nabla y_{\varepsilon}(g + \lambda h) \cdot \nabla g}{|\nabla g|}( \mathbf{x}_0) \right]
\label{eq:2.78}
\end{eqnarray}
where we apply that $g = 0$ on $\partial \Omega$ and $g+ \lambda h = 0$ on  $\partial \Omega_{\lambda}$
and the well known formula for the normal derivative. By some calculations, we get the limit:
$$
M = \left[ -\frac{\partial y_{\varepsilon}(g)}{\partial \mathbf{n}} \frac{\partial h}{\partial \mathbf{n}}\frac{1}{|\nabla g|}
  + \frac{\nabla y_{\varepsilon}(g) \cdot \nabla h}{|\nabla g|}\right]( \mathbf{x}_0).
$$
For the first term above, we have used that:
\begin{equation}\label{eq:2.79}
  \lim_{\lambda \rightarrow 0}\frac{1}{\lambda}
      \left[\frac{1}{|\nabla(g + \lambda h)|} - \frac{1}{|\nabla(g)|}\right]( \mathbf{x}_0)
      = -\frac{\partial h}{\partial \mathbf{n}}( \mathbf{x}_0) \frac{1}{|\nabla g ( \mathbf{x}_0) |^2}.
\end{equation}
Combining (\ref{eq:2.99}) - (\ref{eq:2.79}) with $M$, we infer:

\begin{proposition}\label{p5}
  Assume that  $ \mathcal{F} \subset \mathcal{C}^2$, $f \in H^1 (D)$ and (\ref{eq:1.7}), (\ref{eq:1.11})
  are valid. Then, the  directional directional derivative of the cost (\ref{eq:2.2.1}) with respect
  to functional variations $g+\lambda h$, $g, \; h$ satisfying the above assumptions and
  $g(\mathbf{x}_0) = 0,\; h(\mathbf{x}_0) = 0$ is given by:
\begin{eqnarray}
&&  L + 2M\left(\frac{\partial y_{\varepsilon}(g)}{\partial \mathbf{n}}( \mathbf{x}_0) - \alpha\right) = \nonumber \\
&& 2\left(\frac{\partial y_{\varepsilon}(g)}{\partial \mathbf{n}}( \mathbf{x}_0) - \alpha\right)
\left[\frac{\partial  u}{\partial \mathbf{n}}
- \frac{\partial y_{\varepsilon}(g)}{\partial \mathbf{n}}
  \frac{\partial h}{\partial \mathbf{n}}\frac{1}{|\nabla g|}
+ \frac{\nabla y_{\varepsilon}(g) \cdot \nabla h}{|\nabla g|}\right]( \mathbf{x}_0).
\label{eq:2.80}
\end{eqnarray}
\end{proposition}

A similar calculation can be performed for the pointwise cost functional
(\ref{eq:2.2.2}), term by term. 

\vspace{2mm}
\noindent
Denote by $\mu \in L^2(D)$ the last term in (\ref{3.7}). In fact, at least $H^1_0(D)$ regularity
is valid for $\mu$, under our assumptions, and  $u \in H^3(D)$ depends linearly and continuously
on $\mu$ in these spaces. Then, $\frac{\partial  u}{\partial \mathbf{n}}$ and the first term in
(\ref{eq:2.80}) depend linearly and continuously on $\mu \in H^1_0(D)$.
%Notice that the normal
%derivative in an interior point $ \mathbf{x}_0 \in D$ has a precise meaning (including in
%(\ref{eq:2.81}), (\ref{eq:2.82}) below), namely
%$\frac{\nabla y_{\varepsilon} \cdot \nabla g}{|\nabla g|}( \mathbf{x}_0)$,
%$\frac{\nabla u\cdot \nabla g}{|\nabla g|}( \mathbf{x}_0)$. 
There is a unique $p \in H^{-1}(D)$
such that the linear continuous functional on the left side below, can be expressed as:
\begin{eqnarray}
2\left(\frac{\partial y_{\varepsilon}(g)}{\partial \mathbf{n}}( \mathbf{x}_0) - \alpha\right)
\frac{\partial  u}{\partial \mathbf{n}}( \mathbf{x}_0) = <p,\mu>
= <p,\dfrac{1}{\varepsilon}H^\prime_{\varepsilon}(g)y_{\varepsilon}h>,
\label{eq:2.81}
\end{eqnarray}
where $<\cdot , \cdot>$ denotes the pairing in $ H^{-1}(D)  \times  H^1_0(D)$. We define the
adjoint equation by rewriting (\ref{eq:2.81}) in the form:
\begin{eqnarray}
2\left(\frac{\partial y_{\varepsilon}(g)}{\partial \mathbf{n}}( \mathbf{x}_0) - \alpha\right)
\frac{\partial  u}{\partial \mathbf{n}}( \mathbf{x}_0)
= - <p,-\Delta u + \beta^\prime_{\varepsilon}(y_{\varepsilon})u + \dfrac{1}{\varepsilon}H_{\varepsilon}(g)u>,
\label{eq:2.82}
\end{eqnarray}
where $u \in H^3(D) \cap H^1_0(D)$ has to be considered an arbitrary test function. This approach,
called the transposition method, is due to J.L. Lions, see \cite{JLL}, Ch.2.5.4, where a related
cost functional is considered. Moreover, one adds the null Dirichlet condition to  (\ref{eq:2.82}),
due to the employed class of test functions and a formal interpretation of equation (\ref{eq:2.82}) is:
\begin{eqnarray}
  -\Delta p + \beta^\prime_{\varepsilon}(y_{\varepsilon})p + \dfrac{1}{\varepsilon}H_{\varepsilon}(g)p
  = -P(y_{\varepsilon}) \delta_{\mathbf{x}_0},
\label{adj}
\end{eqnarray}
in $D$, where $\delta_{\mathbf{x}_0}$ is the Dirac distribution concentrated in $\mathbf{x}_0$ and
$P(y_{\varepsilon})$ is generated starting from the linear form in $u$:
\begin{eqnarray}
  2 \int_D \left(\frac{\nabla y_{\varepsilon} \cdot \nabla g}{|\nabla g|} - \alpha\right)
  \frac{\nabla u \cdot \nabla g}{|\nabla g|}dx.
\end{eqnarray}

We have proved:

\begin{proposition}\label{p2.6}
  The adjoint equation associated to the cost functional (\ref{eq:2.2.1}) and the state system
  (\ref{eq:2.4}), (\ref{eq:2.5}) is given by (\ref{eq:2.82}) and it has a unique transposition
  solution $p \in H^{-1}(D)$.
\end{proposition}

In the next section, we shall employ a direct approach based on Prop.\ref{p5}, in the finite
dimensional optimization problem obtained via discretization. Due to the very weak regularity
properties of the adjoint state discussed in Prop.\ref{p2.6}, the usual method to compute the
gradient of the cost via the adjoint system, seems difficult to apply
(but we indicate a regularized variant in one example).

\section{Discretization and numerical tests}

We assume $D\subset \mathbb{R}^2$ to be polygonal domain and $\mathcal{T}_h$ a mesh in $D$
of size $h>0$, each element $T$ of the mesh is a triangle.

We set the $\mathbb{P}_1$ Lagrangian finite element spaces
\begin{eqnarray*}
W_h &=& \{ v_h\in \mathcal{C}(\overline{D});\ \forall T\in \mathcal{T}_h,
\ (v_h)_{|T}\in\mathbb{P}_1(T) \} \\
V_h &=& \{ v_h\in W_h;\ v_h=0\ \mathrm{on}\ \partial D\} .
\end{eqnarray*}

Let $\varphi_h\in W_h$ be an approximation of the obstacle $\varphi$.  We set
\begin{equation}\label{ch3_2:Kh}
  \mathcal{K}_h= \{ v_h \in V_h;\ v_h \geq \varphi_h\ \mathrm{in}\ D \}.
\end{equation}

The discretized state equation, associated to (\ref{eq:2.4}),
  (\ref{eq:2.5}), in variational formulation,  is: find 
$y_h\in \mathcal{K}_h$ such that
\begin{eqnarray}
&&  \int_D \nabla y_h\cdot \nabla (y_h-v_h)\,d\mathbf{x}
  + \frac{1}{\varepsilon} \int_D H_\eta(g_h)y_h(y_h-v_h)\,d\mathbf{x}\nonumber\\
&\leq& \int_D f(y_h-v_h)\,d\mathbf{x}, \quad \forall v_h\in \mathcal{K}_h  \label{ch3_2:VIh}
\end{eqnarray}
with $f\in L^2(D)$.
Here $H_{\eta}\in\mathcal{C}^1(\mathbb{R})$, $\eta>0$, is a regularization of the
Heaviside function, given by
\begin{equation}\label{ch3_2:Hreg}
H_{\eta}(r)=\left\{ 
\begin{array}{cl}
  1, & r> \eta\\
  \frac{(-2r+3\eta)r^2}{\eta^3}, & r\in [0,\eta]\\
  0, & r <0.
\end{array}
\right.
\end{equation}

For the following example of maximal monotone graph
$\beta\subset \mathbb{R} \times \mathbb{R}$, 
\begin{equation}\label{ch3_2:beta}
\beta(r)=\left\{ 
\begin{array}{cl}
	0,           & r>0\\
	]-\infty,0], & r=0\\
	\emptyset  , & r <0
\end{array}
\right.
\end{equation}
we introduce the approximation
$\beta_{\eta,\varepsilon_2}\in\mathcal{C}^1(\mathbb{R})$, ($\varepsilon_2>0$)  given by
\begin{equation}\label{ch3_2:betareg}
\beta_{\eta,\varepsilon_2}(r)=
\left\{ 
\begin{array}{cl}
0,           & r>0\\
r^2\left(\frac{-r}{(\eta)^2\varepsilon_2}
-\frac{2}{\eta\varepsilon_2}\right), &
r\in [-\eta,0]\\
\frac{r}{\varepsilon_2} , & r <-\eta .
\end{array}
\right.
\end{equation}
We take $\eta > \varepsilon$ and $\varepsilon_2 > \varepsilon$
to avoid some numerical difficulties, see \cite{SEMA_2024}.

Each $g_h\in W_h$ can be written as $g_h=\sum_{i=1}^n G_i\phi_i$ where
$(\phi_i)_{1\leq i\leq n}$ is the basis of the hat functions and $n=\mathrm{dim}(W_h)$.
We introduce $\mathcal{J}_1:\mathbb{R}^n \rightarrow \mathbb{R}$ defined by
$$
\mathcal{J}_1(G)=J_1(g_h)
$$
where $G=(G_i)_{1\leq i \leq n}$, see(\ref{eq:2.2.1}).

We denote $I=\{1,\dots,n\}$ and let $I_0\subset I$ be the indices of the vertices
$A_i$ of $\mathcal{T}_h$
such that $\|\mathbf{x}_0-A_i \| < hC$ with $C\geq 2$.
As an example, for the triangulation and $\mathbf{x}_0$ of Test 1, we get $I_0$ of 13 elements
for $C=2$ and 33 elements for $C=3$.

Another choice for $I_0$ is the set of indices of the vertices of the triangle $T\in \mathcal{T}_h$
containing $\mathbf{x}_0$. In this case, the parameter $C$ is not needed.

If $i\in I\setminus I_0$,
then $\mathbf{x}_0$ doesn't belong to the support of $\phi_i$ and
\begin{equation}\label{cond_phij}
  \forall i\in I\setminus I_0,\ \phi_i(\mathbf{x}_0)=0.
\end{equation}

Let $g_h^0\in W_h$ be a finite element approximation of the initial guess $g^0$ in the subsequent
algorithm, such that $g_h^0(\mathbf{x}_0)=0$ and $g_h^0=\sum_{i=1}^n G_i^0\phi_i$.
For $g_h=\sum_{i\in I_0}G_i^0\phi_i + \sum_{i\in I\setminus I_0} R_i\phi_i$,
$R_i\in \mathbb{R}$, we get
$$
\forall R_i\in \mathbb{R},\ i\in I\setminus I_0,\ 
g_h(\mathbf{x}_0)=\sum_{i\in I_0}G_i^0\phi_i(\mathbf{x}_0) =g_h^0(\mathbf{x}_0)=0.
$$
In other words, we can choose freely $R_i$, for all $i\in I\setminus I_0$.

For $i\in I\setminus I_0$, we set the elliptic problem:
find $u_i\in V_h$ such that
\begin{eqnarray}
  &&  \int_D \nabla u_i\cdot \nabla v_h\,d\mathbf{x}
  + \int_D \beta_{\eta,\varepsilon_2}^\prime(y_h-\varphi_h)u_i v_h\,d\mathbf{x}
  + \frac{1}{\varepsilon} \int_D H_\eta(g_h)u_i v_h\,d\mathbf{x}
  \nonumber\\
  &=& -\frac{1}{\varepsilon}\int_D H_\eta^\prime(g_h)y_h \phi_i v_h\,d\mathbf{x},
  \quad \forall v_h\in V_h . \label{u_i}
\end{eqnarray}
For each $i\in I\setminus I_0$, the matrix of the linear system (\ref{u_i}) is the same,
just the right-hand side is different. We can solve effectively by factorization of the matrix,
see \cite{ciarlet1989}, Chap. 4.
We point out that $u_i\in V_h$, so the linear system has $\dim (V_h)$ equations.

\begin{corollary}\label{cor:1}
Let $g_h$ in $W_h$ and we assume that $g_h(\mathbf{x}_0)=0$. Then
\begin{eqnarray}
&&  \partial_i \mathcal{J}_1(G) = \nonumber \\
&& 2\left(\frac{\partial y_h}{\partial \mathbf{n}}( \mathbf{x}_0) - \alpha\right)
\left[\frac{\partial  u_i}{\partial \mathbf{n}}
- \frac{\partial y_h}{\partial \mathbf{n}}
  \frac{\partial \phi_i}{\partial \mathbf{n}}\frac{1}{|\nabla g_h|}
+ \frac{\nabla y_h \cdot \nabla \phi_i}{|\nabla g_h|}\right]( \mathbf{x}_0).
\label{diJ1}
\end{eqnarray}
\end{corollary}

\begin{proof}
For given $g_h$, we set $y_h$ the solution of (\ref{ch3_2:VIh}).
The partial derivative $\partial_i \mathcal{J}_1(G)$ is in fact $J_1^\prime(g_h)\phi_i$.
The last one
can be computed using (\ref{eq:2.80}), where $y_\varepsilon$, $u$, $h$ are replaced with
$y_h$, $u_i$, $\phi_i$ respectively. Here
$\frac{\partial y_h}{\partial \mathbf{n}}(\mathbf{x}_0)$,
$\frac{\partial u_i}{\partial \mathbf{n}}(\mathbf{x}_0)$
$\frac{\partial \phi_i}{\partial \mathbf{n}}(\mathbf{x}_0)$
mean
$$
\nabla y_h(\mathbf{x}_0) \cdot  \frac{\nabla g_h (\mathbf{x}_0)}{|\nabla g_h(\mathbf{x}_0)|},
\quad
\nabla u_i(\mathbf{x}_0) \cdot  \frac{\nabla g_h (\mathbf{x}_0)}{|\nabla g_h(\mathbf{x}_0)|},
\quad
\nabla \phi_i(\mathbf{x}_0) \cdot  \frac{\nabla g_h (\mathbf{x}_0)}{|\nabla g_h(\mathbf{x}_0)|}
$$
respectively.
\end{proof}

The optimization problem to be solved is
\begin{equation}\label{discret_J1}
  \min_{R_i\in \mathbb{R}} \mathbb{J}_1\left((R_i)_{i\in I\setminus I_0}\right)
  =\mathcal{J}_1\left( (G_i^0)_{i\in I_0} , (R_i)_{i\in I\setminus I_0}\right)
\end{equation}
and we have for all $i\in I\setminus I_0$
$$
\partial_i \mathbb{J}_1\left((R_i)_{i\in I\setminus I_0}\right)
=\partial_i \mathcal{J}_1\left( (G_i^0)_{i\in I_0} , (R_i)_{i\in I\setminus I_0}\right).
$$

The steepest descent direction for (\ref{discret_J1}) is given by
\begin{equation}\label{descent_dir_J1}
-\nabla \mathbb{J}_1\left((R_i)_{i\in I\setminus I_0}\right)
\end{equation}
and we can write the \textbf{Algorithm}:

\bigskip

\textbf{Step 1.} Let $g_h^0=\sum_{i=1}^n G_i^0\phi_i$ be the initial guess, such that
$g_h^0(\mathbf{x}_0)=0$. Set $R_i^0=G_i^0$ for $i\in I\setminus I_0$ and put $k=0$.

\textbf{Step 2.} For $i\in I\setminus I_0$, compute $u_i$ using (\ref{u_i}), then compute
$\nabla \mathbb{J}_1(R^k)$ using (\ref{diJ1}), where $R^k=(R_i^k)_{i\in I\setminus I_0}$.

\textbf{Step 3.} Find $\lambda_k >0$ by line search
$$
\lambda_k \in \arg\min_{\lambda >0}
\mathbb{J}_1\left( R^k-\lambda \nabla \mathbb{J}_1(R^k) \right)
$$

\textbf{Step 4.} Update
$$
R^{k+1}=R^k - \lambda_k \nabla \mathbb{J}_1(R^k)
$$

\textbf{Step 5.} We stop the algorithm when
$| \mathbb{J}_1(R^{k+1}) - \mathbb{J}_1(R^k)| < tol$. If not, go to \textbf{Step 2.}

\bigskip

We notice that the values at the vertices of index $i\in I_0$ of the
parameterization function
$g_h^k=\sum_{i\in I_0}G_i^0\phi_i + \sum_{i\in I\setminus I_0} R_i^k\phi_i$ are fixed,
in order to satisfy the condition $g_h^k(\mathbf{x}_0)=0$, for all $k\geq 0$.
In fact, $g_h^k$ is fixed in a neighborhood of $\mathbf{x}_0$.

\bigskip

\textbf{Test 1.}

The domain is $D=]0,1[\times ]0,1[$,
the obstacle is $\varphi:D\rightarrow\mathbb{R}$, $\varphi=-0.5$, 
$f:D\rightarrow\mathbb{R}$, $f=-100$ and $\alpha=0$.
We use a mesh for $D$ of 26870 vertices, 53138 triangles and size $h=1/150$.    
The following values have been used: penalization parameter $\varepsilon=0.0001$,
parameters for smoothed Heaviside and $\beta$ are $\eta=0.05$, $\varepsilon_2=0.01$.
The starting domain is a disk of radius $r_0=0.25$ and center $(0.5, 0.5)$,
while $\mathbf{x}_0=(0.25,0.5)$. 
For the stopping criterion, we use $tol=10^{-6}$.
For the cases a), b), c) we use the descent direction $-\nabla \mathbb{J}_1$
and for the case d) we use the simplified (partial) direction $-y_hp_h$.

\medskip
\textbf{Case a). We use $\alpha=0$ and $C=2$;  $I_0$ has 13 vertices.}

The steepest descent direction algorithm stops after 5 iterations and
the objective functions decreases from $36.82$ to $2.80\times 10^{-7}$,
see Figure \ref{fig:k2_kf150_J1}.
The evolution of $\Omega$ is prezented in Figure \ref{fig:k2_kf150_Omega}.
The boundary of the final $\Omega$ is plotted in  Figure \ref{fig:k2_kf150_iso}.

\begin{figure}[ht]
\centering
  \includegraphics[width=6cm]{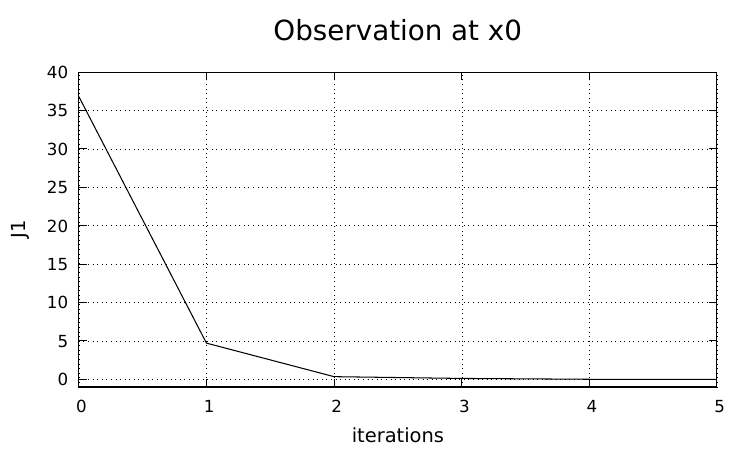}
  \caption{Test 1, case a). The history of the objective function $J_1$.}
\label{fig:k2_kf150_J1}
\end{figure}

\begin{figure}[ht]
\centering
\includegraphics[width=6cm]{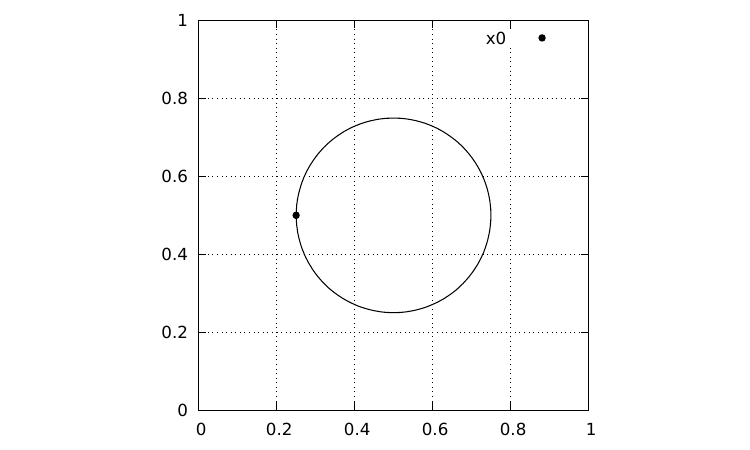}
\includegraphics[width=6cm]{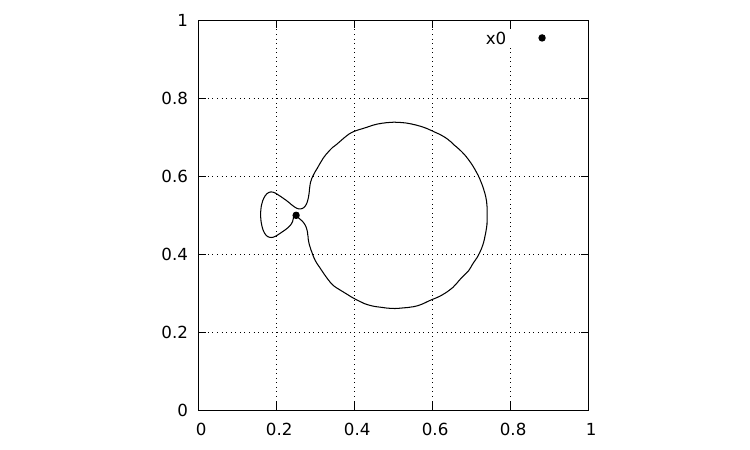}
\caption{Test 1, case a). The position of $\mathbf{x}_0$ and the
  boundaries of $\Omega$: initial (left)
    and final (right).}
\label{fig:k2_kf150_iso}
\end{figure}
\newpage

\begin{figure}[ht]
\centering
  \includegraphics[width=6cm]{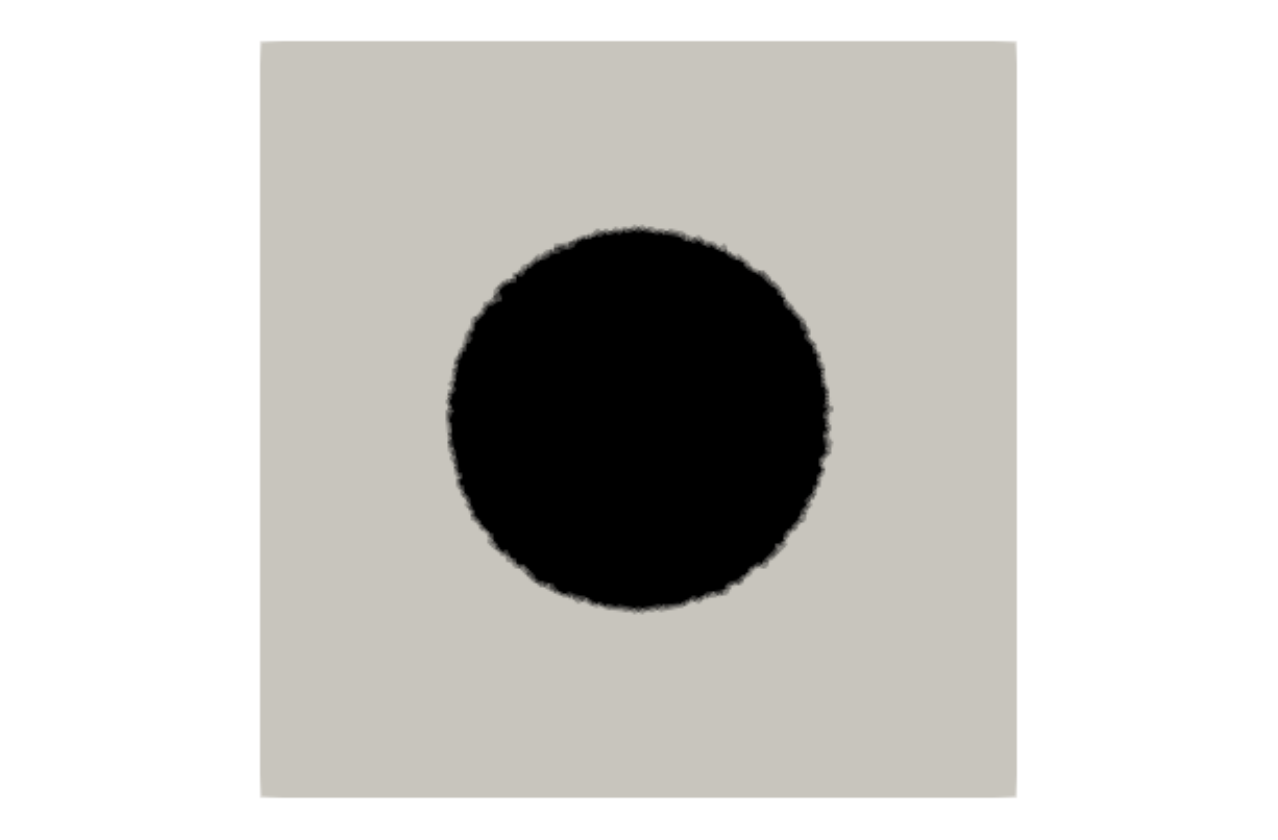}\quad
  \includegraphics[width=6cm]{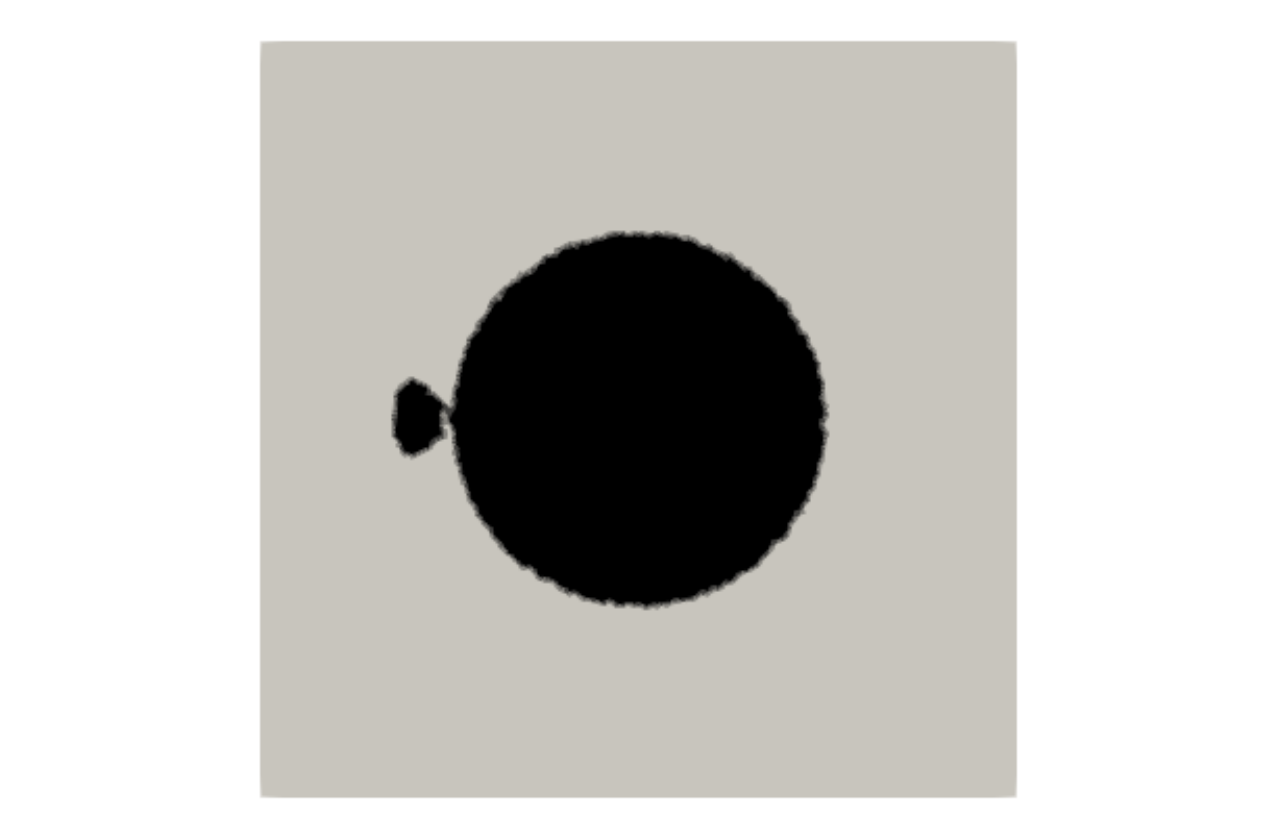}\\
  \includegraphics[width=6cm]{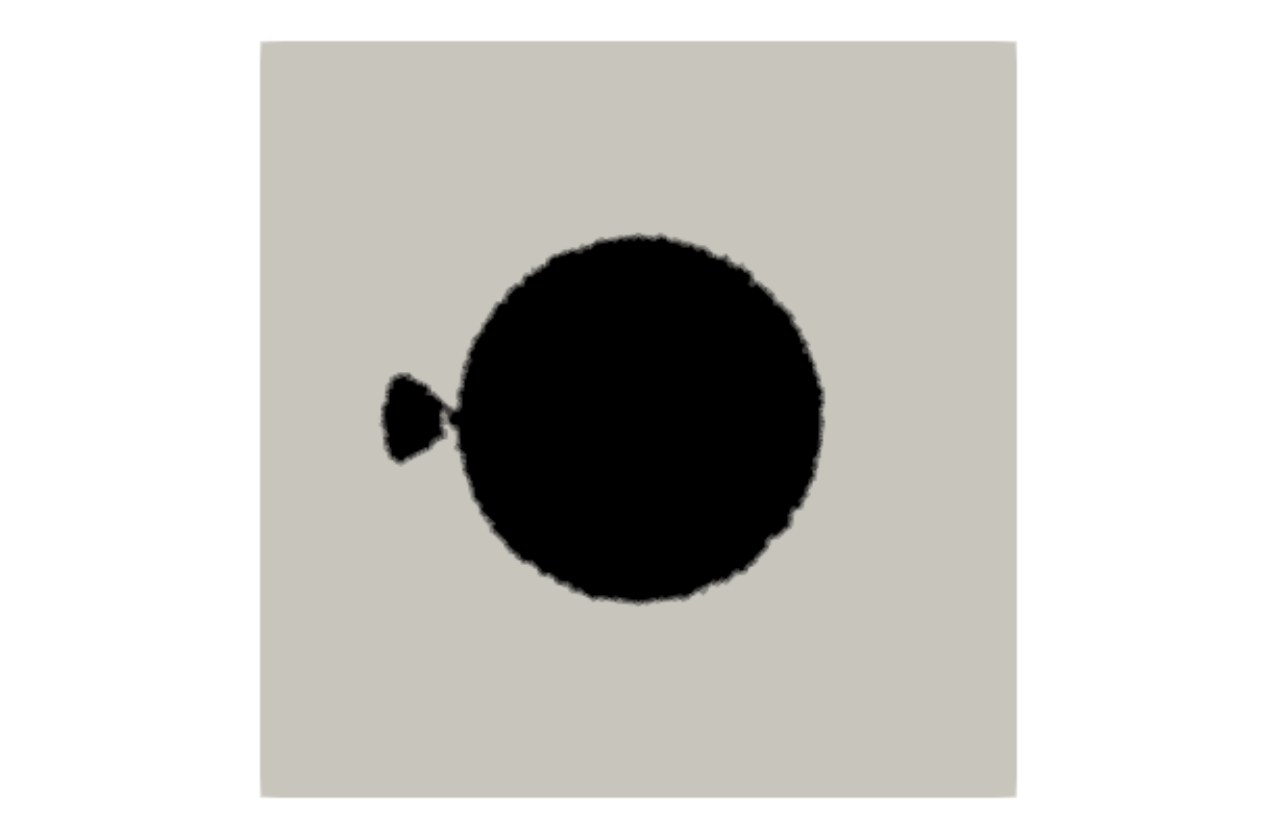}\quad
  \includegraphics[width=6cm]{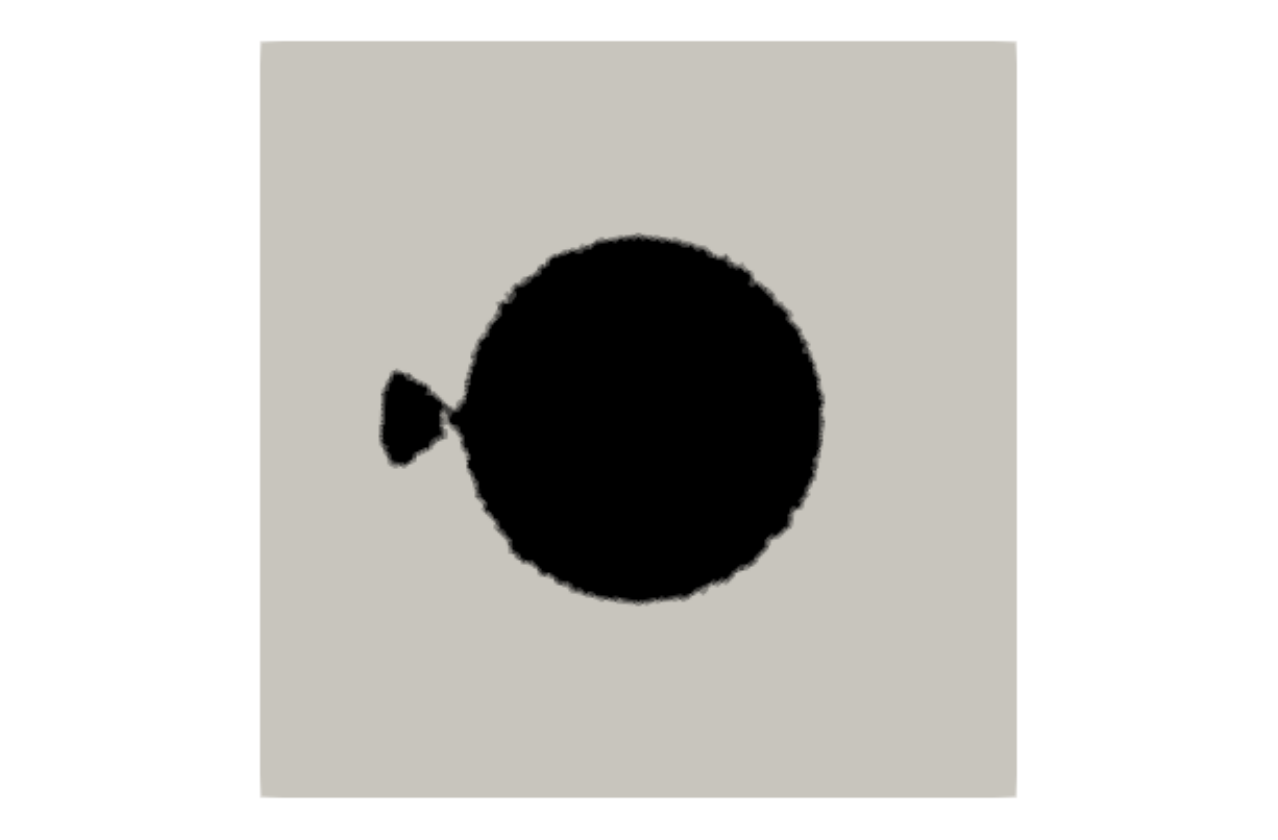}\\
  \includegraphics[width=6cm]{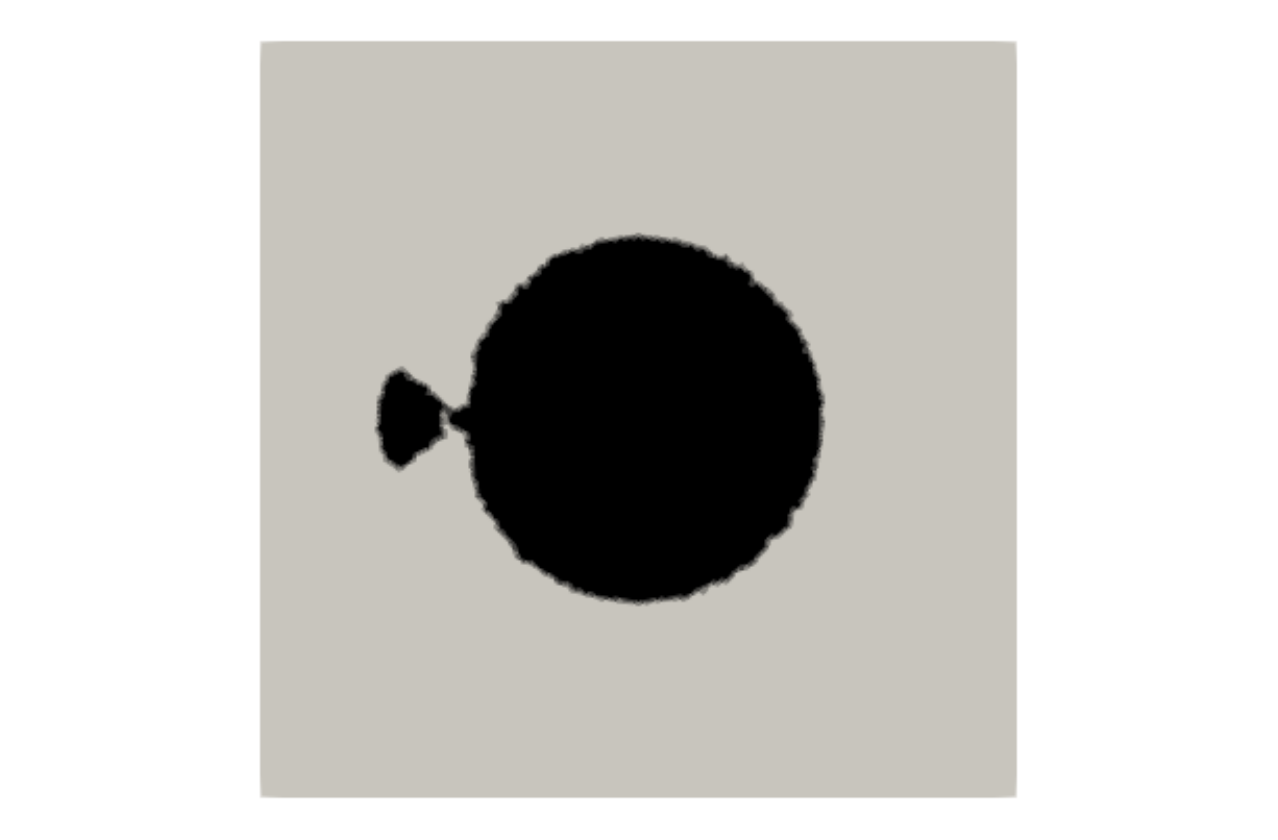}\quad
  \includegraphics[width=6cm]{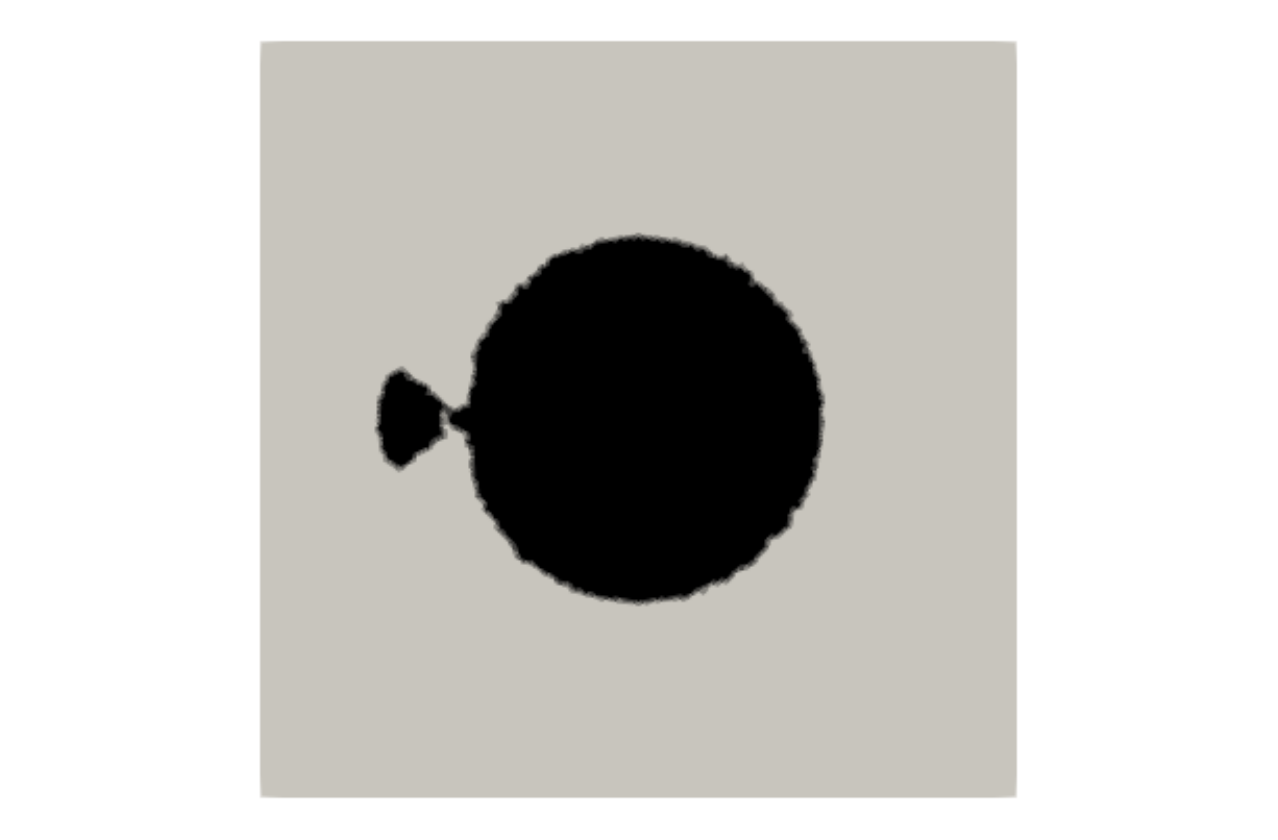}
  \caption{Test 1, case a). $\Omega$ at iterations: 0, 1, 2, 3, 4, 5
      (from left to right and from top to the bottom).
}
\label{fig:k2_kf150_Omega}
\end{figure}

In Figure \ref{fig:k2_kf150_yh}, we show the final state solution $y_h$.
At the final iteration,
we obtained 
$$
\nabla y_h(\mathbf{x}_0)=(-0.001926, -0.367164),\quad
\frac{\nabla g_h (\mathbf{x}_0)}{|\nabla g_h(\mathbf{x}_0)|}=
(-0.999978, 0.006689)
$$
then $\frac{\partial y_h}{\partial \mathbf{n}}(\mathbf{x}_0)-\alpha=-0.000529
= \frac{\partial y_h}{\partial \mathbf{n}}(\mathbf{x}_0)$.

In Figure \ref{fig:k2_kf150_yh}, we see that the obtained optimal geometry ``generates'' a saddle point
for the optimal $y_h$ and $\mathbf{x}_0$ is close to this saddle point. That's why the final cost has
such a low value. Notice however that the norm of $\nabla y_h(\mathbf{x}_0) \neq 0$, that is the
tangential derivative of the optimal state in $\mathbf{x}_0$ is not zero and the Dirichlet
boundary condition for $y_h$ is slightly violated due to the numerical errors
(but the numerical minimization
is correctly performed).

\begin{figure}[ht]
\centering
\includegraphics[width=10cm]{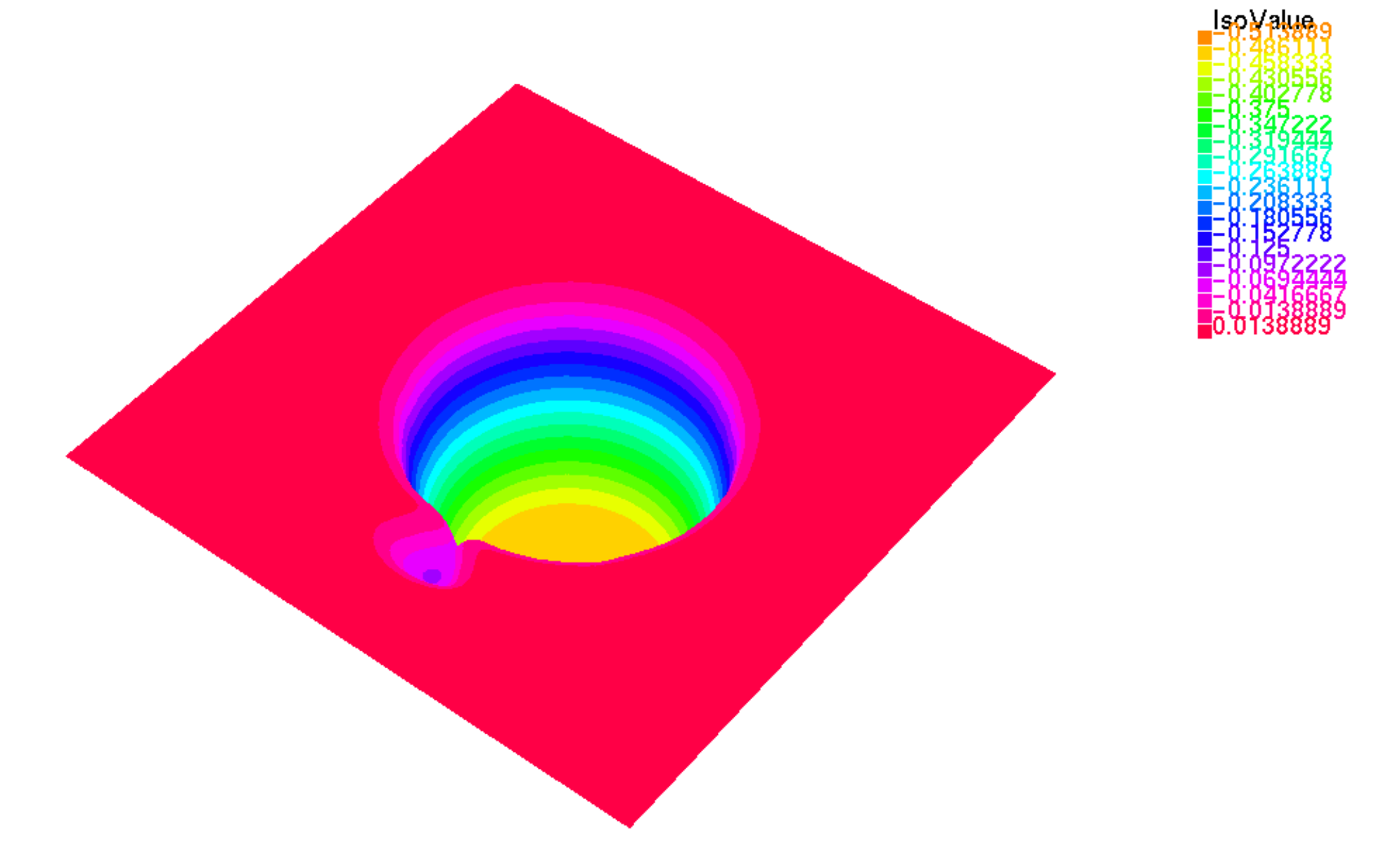}
\caption{Test 1, case a). The final state solution $y_h$. For the computed solution,
  we have $\max (y_h)=0$ and $\min (y_h)=-0.5$. The visualization software adds at
  the legend one level
greater than max and one level smaller than min.}
\label{fig:k2_kf150_yh}
\end{figure}

\medskip
\textbf{Case b). We use $\alpha=0$ and $I_0$ with 3 vertices.}

If $I_0$ is the set of indices of the vertices of the triangle $T\in \mathcal{T}_h$
containing $\mathbf{x}_0$,
the algorithm stops after 3 iterations, the values of the objective functions are:
$36.82$ (initial), $3.40515$, $0.000325$,
$7.05\times 10^{-7}$ (final).

\begin{figure}[ht]
\centering
\includegraphics[width=6cm]{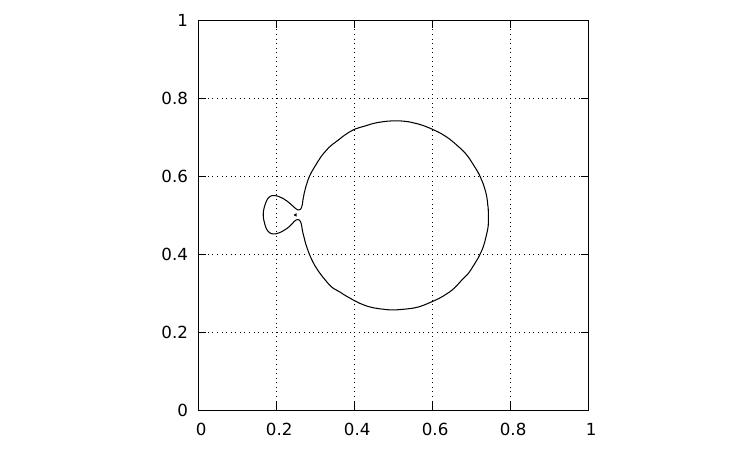}\quad
\includegraphics[width=6cm]{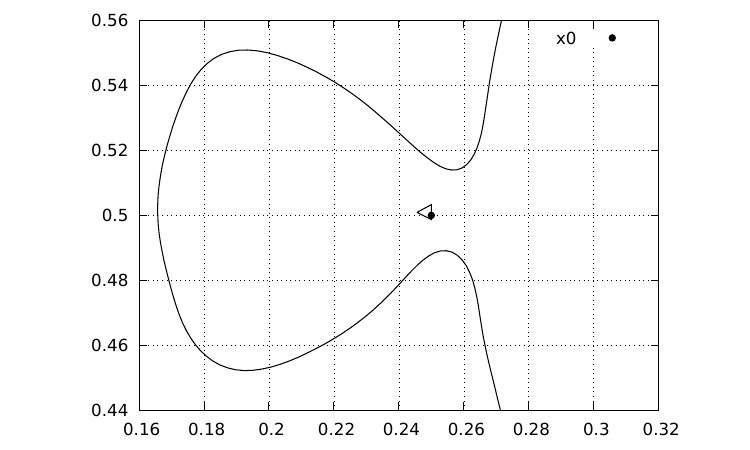}  
\caption{Test 1, case b).
  The final $\partial \Omega$ (left) and zoom around $\mathbf{x}_0$ (right).}
\label{fig:I0_triang}
\end{figure}

\newpage
We see in Figure \ref{fig:I0_triang} that the final $\partial \Omega$ has
two connected boundaries (topological optimization), but one is very small.
It is a triangular polygonal line containing $\mathbf{x}_0$.
The triangle, having as boundary the polygonal line, represents a hole for the optimal $\Omega$.
We point out that this triangle is not an element of the triangulation, otherwise we get
$g_h=0$ in this triangle, which contradicts (\ref{eq:1.7}).

At the final iteration, we have
$$
\nabla y_h(\mathbf{x}_0)=(-0.004556, -0.806676),\quad
\frac{\nabla g_h (\mathbf{x}_0)}{|\nabla g_h(\mathbf{x}_0)|}=
(-0.999978, 0.006689)
$$
then $\frac{\partial y_h}{\partial \mathbf{n}}(\mathbf{x}_0)-\alpha=-0.000839
= \frac{\partial y_h}{\partial \mathbf{n}}(\mathbf{x}_0)$.

\medskip
\textbf{Case c). We use $\alpha=1$ and $C=2$; $I_0$ has 13 elements.}

The algorithm stops after 3 iterations,
the values of the objective functions are: $25.69$ (initial), $1.45939$, $0.000395$,
$1.26\times 10^{-7}$ (final). 

\begin{figure}[ht]
\centering
\includegraphics[width=6cm]{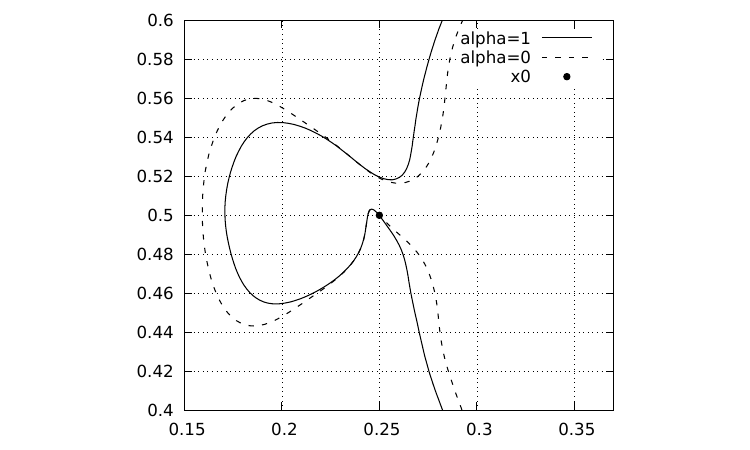}\quad
\includegraphics[width=5cm]{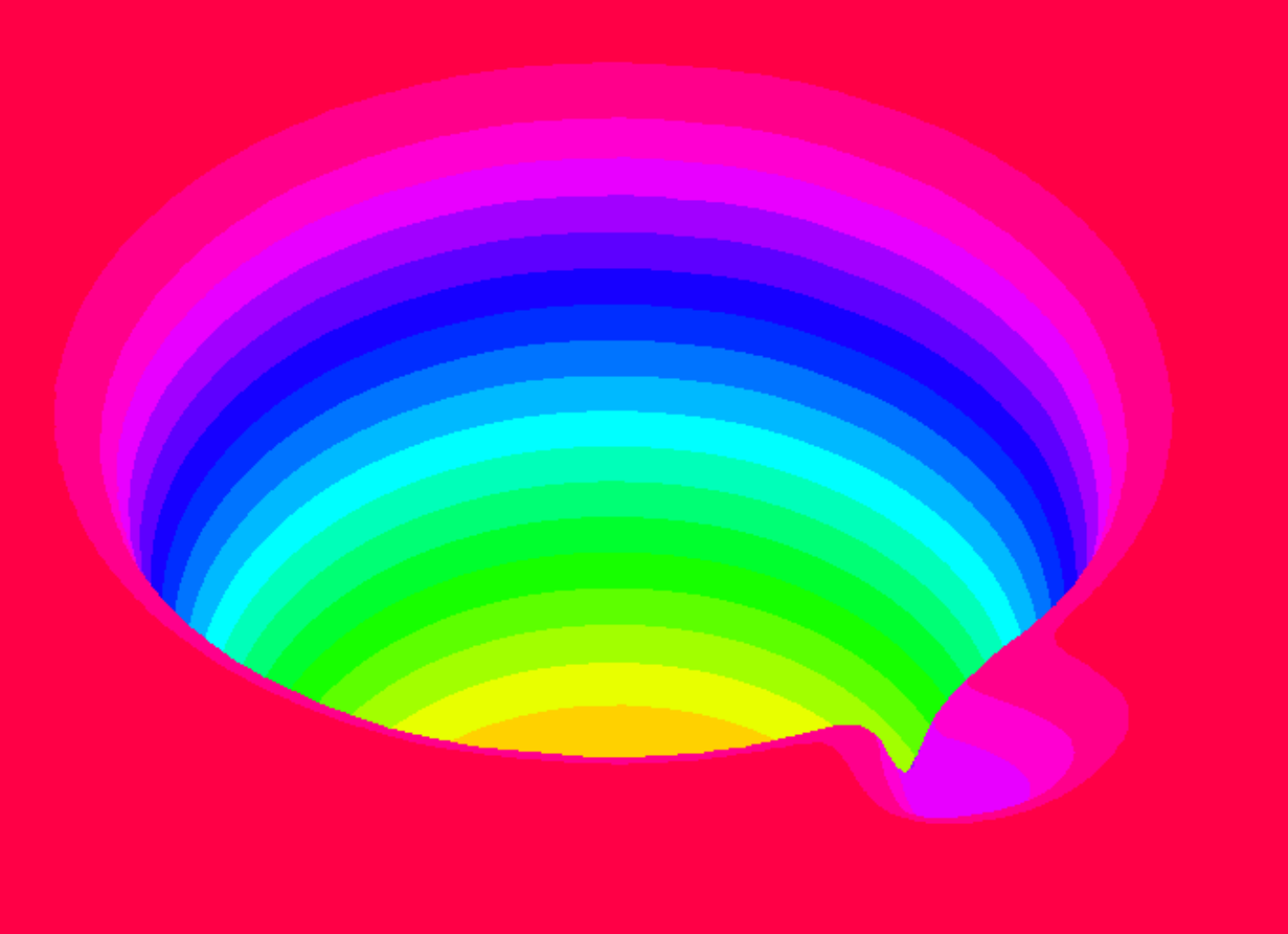}
\caption{Test 1, case c).
  The final $\partial \Omega$ (left, zoom) and the final state $y_h$ (right, zoom).}
\label{fig:alpha_1}
\end{figure}

At the final iteration, we obtained 
$$
\nabla y_h(\mathbf{x}_0)=(-1.00285, -0.475403),\quad
\frac{\nabla g_h (\mathbf{x}_0)}{|\nabla g_h(\mathbf{x}_0)|}=
(-0.999978, 0.006689)
$$
then $\frac{\partial y_h}{\partial \mathbf{n}}(\mathbf{x}_0)-\alpha=-0.000356$.
Some images for the final solution are plotted in Figure \ref{fig:alpha_1}.
Although in Figure \ref{fig:alpha_1}, we also see  a saddle point for $ y_h$, $\mathbf{x}_0$ is sufficiently
far from it such that  $\frac{\partial y_h}{\partial \mathbf{n}}(\mathbf{x}_0) -1 = -0.000356$.

We see that $\frac{\nabla g_h (\mathbf{x}_0)}{|\nabla g_h(\mathbf{x}_0)|}$ is the same for the three
cases a), b), c). It is correct, because $g_h$ is not modified during the algorithm in the vertices
$i\in I_0$. We have that $\mathbf{x}_0$ belongs to $\cup_{i\in I_0} \mathrm{supp} (\phi_i)$.
The set $I_0$ in the case b) is included in the set $I_0$ of cases a) and c).
The triangulation, $\mathbf{x}_0$ and the initial $g_h$ are the same for the three cases.

The holdall domain $D$ is defined by the user. For rectangular shape, we can built a mesh
composed only of right triangles or rectangular elements. Here, we use an unstructured triangulation
(irregular grid) which can be employed also for non rectangular $D$.

\medskip
\textbf{Case d). We use $\alpha=0$, $C=2$ and $-y_hp_h$ as descent direction.}

We can increase the regularity in the adjoint equation (\ref{adj}) by using
$\zeta_{\varepsilon_1}(\mathbf{x}-\mathbf{x}_0)$, a mollifier approximation of the Dirac
functional $\delta_{\mathbf{x}_0}$.
We write $\zeta_{\varepsilon_1}(\mathbf{x})
= \frac{1}{\varepsilon_1^2}\zeta\left( \frac{\mathbf{x}}{\varepsilon_1} \right)$, where
$\zeta\in \mathcal{C}^\infty(\mathbb{R}^2)$ has support in the unitary disk centered at $(0,0)$
and $\int_{\mathbb{R}^2} \zeta(\mathbf{x})\, d\mathbf{x}=1$, $\varepsilon_1>0$.

An approximate version of the adjoint state equation  (\ref{adj}) is: find $p\in H_0^1(D)$ such that
\begin{eqnarray}
  &&  \int_D \nabla p\cdot \nabla v\,d\mathbf{x}
  + \int_D \beta_{\varepsilon}^\prime(y_{\varepsilon})p v\,d\mathbf{x}
  + \frac{1}{\varepsilon} \int_D H_\varepsilon(g)p v\,d\mathbf{x}
  \nonumber\\
  &=& -\int_D 
2\left(\frac{\partial y_{\varepsilon}}{\partial \mathbf{n}}(\mathbf{x}) - \alpha\right)
\frac{\partial v}{\partial \mathbf{n}}(\mathbf{x})  
\zeta_{\varepsilon_1}(\mathbf{x}-\mathbf{x}_0)\,d\mathbf{x},
\quad \forall v\in H_0^1(D).
\label{p_reg}
\end{eqnarray}

Taking $v=u$ in (\ref{p_reg}) and using (\ref{3.7}), we get that
$$
\frac{1}{\varepsilon}\int_D H_\varepsilon^\prime(g)h y_{\varepsilon} p \,d\mathbf{x}
=\int_D 
2\left(\frac{\partial y_{\varepsilon}}{\partial \mathbf{n}}(\mathbf{x}) - \alpha\right)
\frac{\partial u}{\partial \mathbf{n}}(\mathbf{x})  
\zeta_{\varepsilon_1}(\mathbf{x}-\mathbf{x}_0)\,d\mathbf{x}
$$
and the right-hand side is an approximation of L defined by (\ref{eq:2.99}).
Since $H_\varepsilon^\prime \geq 0$, the left-hand side is negative, for $h=-y_{\varepsilon} p$ and can
be used in the gradient-type algorithms, inspired by \cite{MT2022}, although it has
a ``partial'' character.

All the data are the same as in the case a), but here we use the discretization (\ref{ph})
and in the Algorithm, the simplified descent direction
$-y_hp_h$ instead of $-\nabla \mathbb{J}_1$. That is, we use just a part of the
gradient computed in the previous section (but the descent properties are preserved according
to Fig. \ref{fig:test1_d_J1}).
The objective functions decreases from $36.82$ to $1.15\times 10^{-7}$, see Figure \ref{fig:test1_d_J1}.
The Algorithm stops after 4 iterations.

\begin{figure}[ht]
\centering
  \includegraphics[width=6cm]{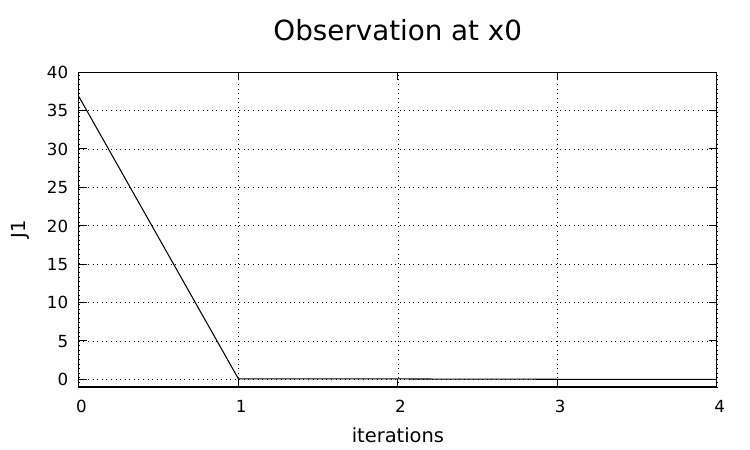}
  \caption{Test 1, case d). The history of the objective function.}
\label{fig:test1_d_J1}
\end{figure}

The finite element approximation of the adjoint state is: find $p_h \in V_h$ such that
\begin{eqnarray}
&&  \int_D \nabla p_h\cdot \nabla v_h\,d\mathbf{x}
+ \int_D \beta_{\eta,\varepsilon_2}^\prime(y_h-\varphi_h) p_h v_h\,d\mathbf{x}
+ \frac{1}{\varepsilon} \int_D H_\eta(g_h)p_h v_h\,d\mathbf{x}
\nonumber\\
&=& -\int_D 
2\left(\frac{\partial y_h}{\partial \mathbf{n}}(\mathbf{x}) - \alpha\right)
\frac{\partial v_h}{\partial \mathbf{n}}(\mathbf{x})  
\zeta_{\varepsilon_1}(\mathbf{x}-\mathbf{x}_0)\,d\mathbf{x},
\quad \forall v_h\in V_h.
\label{ph}
\end{eqnarray}
For the mollifier  approximation of the Dirac function, we use $\varepsilon_1=0.05$.

\begin{figure}[ht]
\centering
\includegraphics[width=6cm]{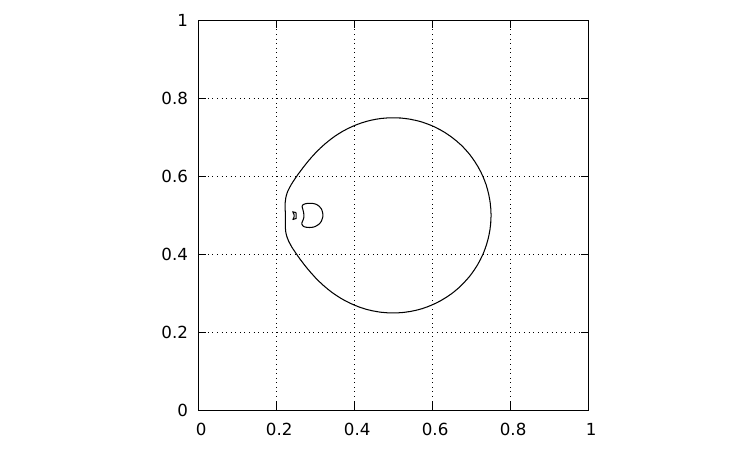}
\includegraphics[width=6cm]{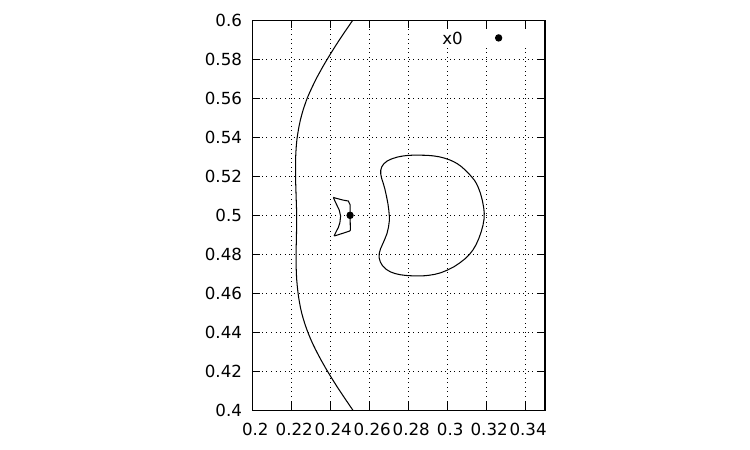}
\caption{Test 1, case d).
The final boundary of $\Omega$  (left),
the position of $\mathbf{x}_0$ and the
final boundary (right, zoom).}
\label{fig:test1_d_final_iso}
\end{figure}

We notice in Figure \ref{fig:test1_d_final_iso} that the final $\partial \Omega$ has
three connected components of the boundary (topological optimization) and the smallest one
is containing $\mathbf{x}_0$.
That is, the final $\Omega$ has two holes.

\bigskip
\textbf{Test 2.}

We can use  similar computations, when we have pointwise observations at a finite number
of fixed points $\mathbf{x}_0^j$, $1\leq j \leq \ell$. We assume that $g(\mathbf{x}_0^j)=0$, for all
$j$. The objective function is the sum of the expression (\ref{eq:2.2.1}) computed at each
point $\mathbf{x}_0^j$.
The set $I_0$ is, in this case, the indices of the vertices $A_i$ of $\mathcal{T}_h$
such that there exists $1\leq j \leq \ell$, $\|\mathbf{x}_0^j-A_i \| < hC$ with $C\geq 2$.
The descent direction is the sum of the terms in (\ref{descent_dir_J1}) computed at each
point $\mathbf{x}_0^j$.

We have $\ell=3$, the observation points are:
\begin{eqnarray*}
\mathbf{x}_0^1 &=& \left(0.5 + r_0\cos(\pi), 0.5 + r_0\sin(\pi)\right),\\
\mathbf{x}_0^2 &=& \left(0.5 + r_0\cos(\pi-\frac{\pi}{6}), 0.5 + r_0\sin(\pi-\frac{\pi}{6})\right),\\
\mathbf{x}_0^3 &=& \left(0.5 + r_0\cos(\pi+\frac{\pi}{6}), 0.5 + r_0\sin(\pi+\frac{\pi}{6})\right).
\end{eqnarray*}

For $C=3$ we get $I_0$ of 96 elements.
The others parameters are the same as in Test 1 and $\alpha=0$.
The algorithm stops after 5 iterations, the objective functions decreases from
$326.12$ to $1.64 \times 10^{-5}$,
see Figure \ref{fig:3_points_J1}.

In Figure \ref{fig:3_points_Omega_5} we plot the initial and final $\partial\Omega$ and
in Figure \ref{fig:3_points_yh} we show the final state solution $y_h$.
Final $\Omega$ seems to be not connected, while the support of $y_h$ is connected.
This phenomena appears since in the state variational inequality we use in the penalization term
$H_\eta(g_h)$ which is an approximation of the characteristic function of the final $\Omega$,
that is the approximations yield important perturbations on the result without
affecting its essence.

\begin{figure}[ht]
\centering
  \includegraphics[width=6cm]{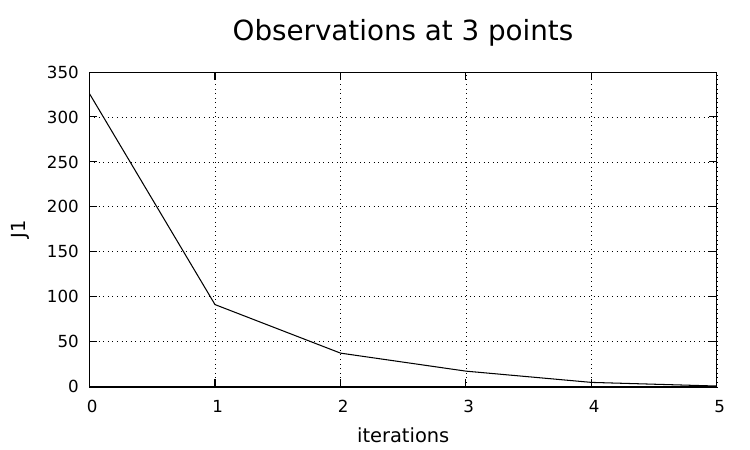}
  \caption{Test 2. The history of the objective function.}
\label{fig:3_points_J1}
\end{figure}

\begin{figure}[ht]
\centering
\includegraphics[width=6cm]{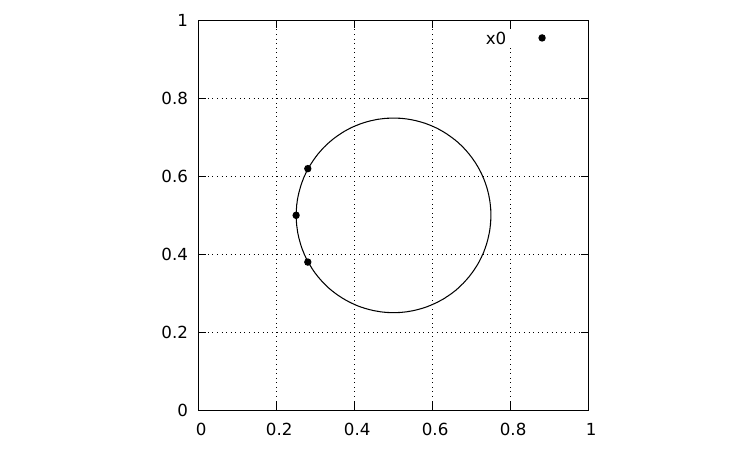}
\includegraphics[width=6cm]{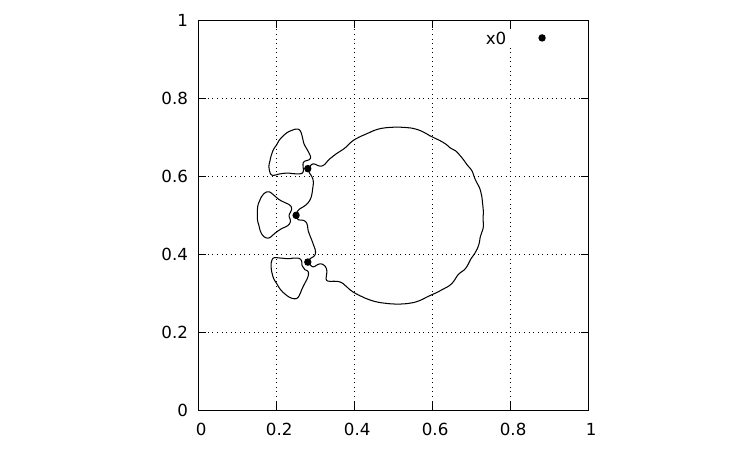}
  \caption{Test 2.The boundaries of $\Omega$: initial (left)
    and final (right).}
\label{fig:3_points_Omega_5}
\end{figure}

\newpage
\begin{figure}[ht]
\centering
\includegraphics[width=6cm]{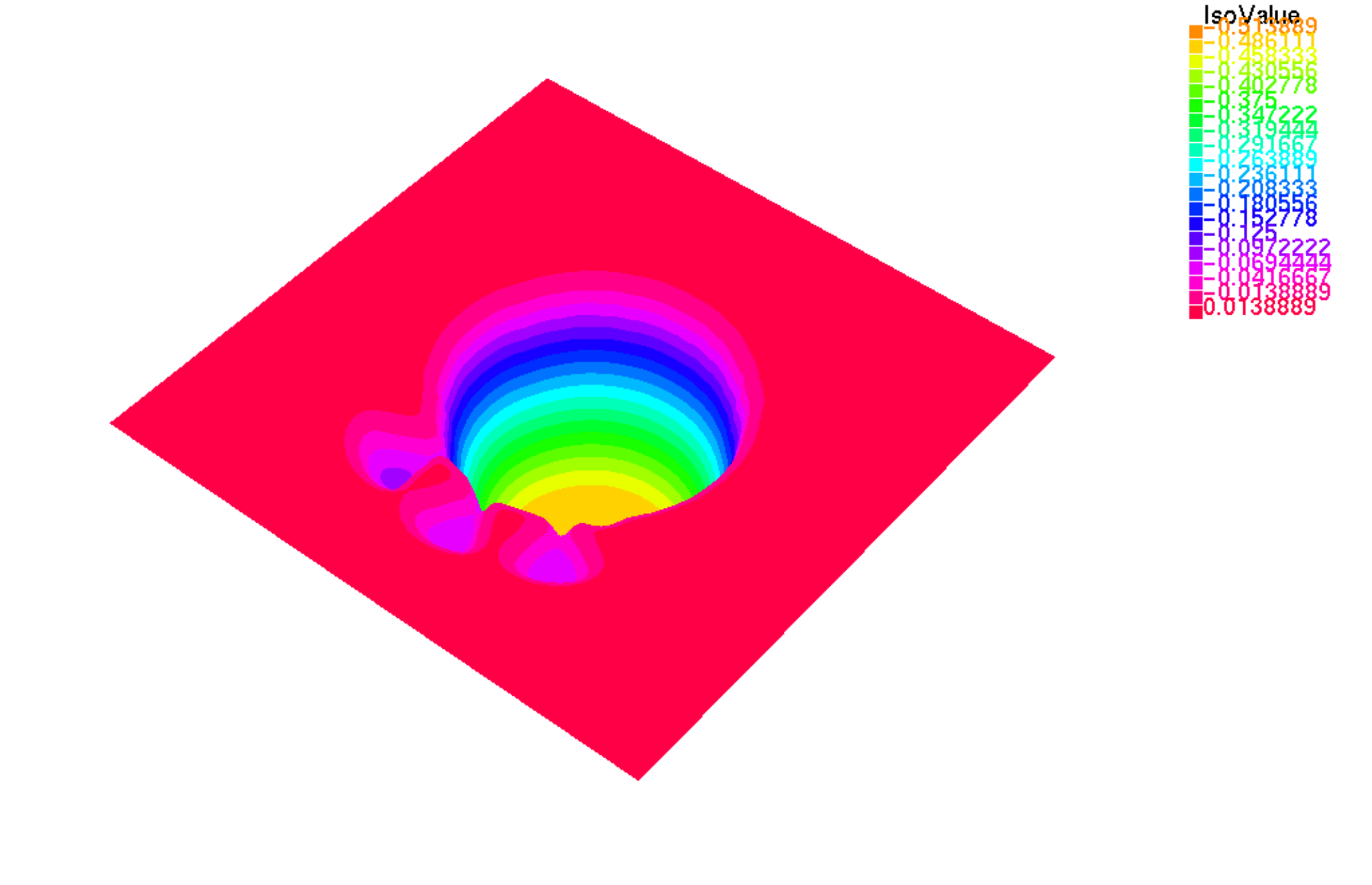}\quad
\includegraphics[width=6cm]{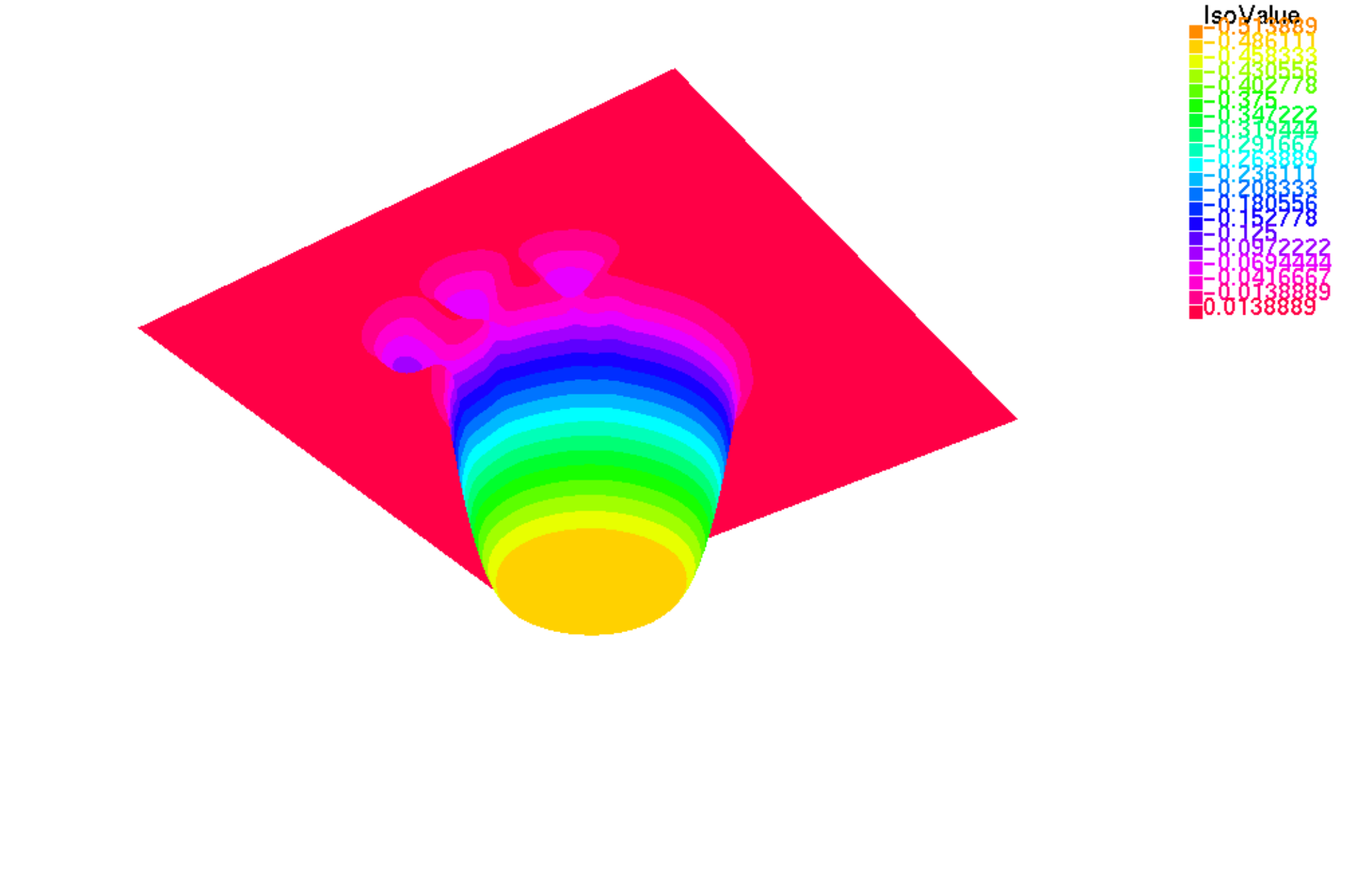}
  \caption{Test 2. The final state solution $y_h$, top view (left) and bottom view (right).}
\label{fig:3_points_yh}
\end{figure}

\end{document}